\documentclass[11pt,reqno]{amsart}
\usepackage{graphicx}
\usepackage{float}
\usepackage{amsfonts}
\usepackage[caption = false]{subfig}
\usepackage{amsmath}
\usepackage{graphics}
\usepackage{float}
\usepackage{enumerate}
\usepackage{mathtools}
\usepackage{amsthm}
\usepackage[english]{babel}
\usepackage{amsfonts,amssymb}% insert here the call for the packages your document requires
\usepackage{mathrsfs}
\usepackage[utf8x]{inputenc}\usepackage[english]{babel}
\usepackage{soul}
\usepackage[all]{xy,xypic}
\usepackage{cite}
\usepackage{relsize}
\numberwithin{equation}{section}
\newtheorem{theorem}{Theorem}[section]
\newtheorem{corollary}[theorem]{Corollary}
\newtheorem{lemma}[theorem]{Lemma}

\theoremstyle{definition}

\theoremstyle{remark}
\newtheorem{remark}[theorem]{Remark}

\numberwithin{equation}{section}
\DeclareMathOperator{\RE}{Re}

\begin{document}
	%\fontsize{14pt}{16pt}\selectfont
	
	\title[\tiny{Coefficient estimates of starlike functions associated with a Petal Shaped Domain}]{Coefficient problems for starlike functions associated with a petal shaped domain}
	
	\author[S. Sivaprasad Kumar]{S. Sivaprasad Kumar}
	\address{Department of Applied Mathematics, Delhi Technological University, Delhi--110042, India}
	\email{spkumar@dce.ac.in}

    	\author{Neha Verma*}
	\address{Department of Applied Mathematics, Delhi Technological University, Delhi--110042, India}
	\email{nehaverma1480@gmail.com}

	
	\maketitle
	\begin{abstract}
In the present investigation, we consider a subclass of starlike functions associated with a petal shaped domain, recently introduced and defined by 
$$\mathcal{S}^{*}_{\rho}:=\{f\in \mathcal{A}:zf'(z)/f(z) \prec 1+\sinh^{-1} z\}.$$ We establish certain coefficient related problems such as sharp first five coefficient bounds along with sharp second and third order Hankel determinants for $\mathcal{S}^{*}_{\rho}$. Also, sixth and seventh coefficient bounds are estimated to obtain the fourth Hankel determinant bound for the same class.
%such as $1+\beta zp'(z)/p^{k}(z), p(z)+\beta zp'(z)/p^{k}(z)\prec 1+\arctan(z)$ 

%$p(z)\prec (1+Az)/(1+Bz)$ for the class $\mathcal{S}^{*}_{\tau}$, several sharp findings on $\beta\in \mathbb{R}$ have been achieved with respect to $\mathcal{S}^{*}_{\tau}$.	
\end{abstract}
\maketitle
	
\section{Introduction}
\label{intro}
Coefficient problems came into the limelight due to the famous Bieberbach's conjecture~\cite{symposium}, which states that  $ |a_n|\leq n,$ holds for every $n>1$ for each univalent function $f$ having the form
\begin{equation}
f(z) = z+\sum_{n=2}^{\infty}a_nz^n.\label{form}
\end{equation}  Then in 1984, the conjecture was proved by Louis de Branges \cite{de branges},
which paved the path for establishing many coefficient related problems till date. For more information about the conjecture, see \cite{symposium}. 
In the mean time, several techniques and subclasses of normalized univalent functions, denoted by $\mathcal{S}$, were evolved. In particular, $\mathcal {S}^*$ and $\mathcal{C}$ are the classes consisting of starlike and convex functions, mapping the unit disc $\mathbb{D}:=\{z:|z|<1\}$ onto a starlike and convex domain respectively, came into existence.

Let $g_1$ and $g_2$ be two analytic functions. We say $g_1$ is subordinate to $g_2$ or symbolically $g_1\prec g_2$, if there exists an analytic function $w$ satisfying $w(0)=0$ and $|w(z)|\leq |z|$ such that  $g_1(z)=g_2(w(z))$. In 1992, Ma-Minda \cite{ma-minda} introduced the following class by unifying all subclasses of starlike functions:
\begin{equation}
   \mathcal{S}^{*}(\varphi)=\bigg\{f\in \mathcal {A}:\dfrac{zf'(z)}{f(z)}\prec \varphi(z) \bigg\},\label{mindaclass}
\end{equation}
where $\mathcal{A}$ contains all normalized analytic functions, which are of the form, given by equation (\ref{form}) and $\varphi$ is an analytic univalent function satisfying $\RE\varphi(z)>0$, $\varphi(\mathbb{D})$ starlike with respect to $\varphi(0)=1$, $\varphi'(0)>0$ and the domain $\varphi(\mathbb{D})$ is symmetric about the real axis.
Note that the sharp upper bounds for functions belonging to the Ma-Minda classes are known only up to $a_4$ and partially for $a_5$.
The sharp bound of $a_2$ was obtained by Ma-Minda \cite{ma-minda}. Ali et al. \cite{pvalentali} obtained the sharp bounds for $a_3$ and $a_4$. One can refer \cite{shelly, bellv} for the fifth coefficient bound of functions in some specific Ma-Minda classes. In the present scenario, obtaining the improved higher-order coefficient bounds for the general Ma-Minda class is a challenge in itself. In fact, the known bounds of Carathe\'{o}dory functions play a major role in the computation of various coefficient bound estimate problems.
We denote the Carathe\'{o}dory class by $\mathcal{P}$, is the class of functions which are analytic on $\mathbb{D}$ carrying the expansion of the form $p(z)=1+\sum_{n=1}^{\infty}p_n z^n$ with positive real part.
The $q^{th}$ Hankel determinants of Taylor's coefficients for functions $f \in \mathcal{A}$ having the form (\ref{form}) assuming $a_ 1:= 1$ are defined as $H_{q,n}(f)$ for $n,q \in \mathbb{Z}^{+}$.
\begin{equation*}
	H_{q,n}(f) =\begin{vmatrix}
a_n&a_{n+1}& \ldots &a_{n+q-1}\\
a_{n+1}&a_{n+2}&\ldots &a_{n+q}\\
\vdots& \vdots &\ddots &\vdots\\
a_{n+q-1}&a_{n+q}&\ldots &a_{n+2q-2}
\end{vmatrix},
\end{equation*}
 was initially considered in \cite{pomi} and is still a topic of great interest to many authors. For different choices of $q$ and $n$, the generalized Fekete-Szeg\"{o} functional is given as $a_3-\mu a_2^2$, where $\mu$ is a complex number, and its particular case is $a_3-a_2^2$.
The expressions for second and third order Hankel determinants are given by
\begin{equation}
   H_{2,2}(f)=a_2a_4-a_3^2,\label{h22}
\end{equation}
\begin{equation*}
    H_{2,3}(f)=a_3a_5-a_4^2,
\end{equation*}
and
\begin{equation}
	H_{3,1}(f) =2 a_2a_3a_4-a_3^3-a_4^2-a_2^2a_5+a_3a_5.\label{1h3}
\end{equation}

In the study of second Hankel determinants, Janteng et al. \cite{2 Janteng}
determined the sharp estimates of $H_{2,2}(f)$ for functions in the classes $\mathcal{S}^{*}$ and $\mathcal{C}$ as $1$ and $1/8$ respectively. Krishna et al. \cite{krishna bezilevic} obtained the sharp bounds of $H_{2,2}$ for the class of Bazilevi$\breve{c}$ functions. In fact, the estimation of the bound of Hankel determinant for $q=3$ is more challenging than $q=2$ and there are only a few sharp bounds of Hankel determinants known till now. Zaprawa \cite{zap} estimated that \begin{equation*}
    |H_{3,1}(f)|\leq \begin{cases}
    1,& f\in \mathcal{S}^{*}\\
    49/540,& f\in \mathcal{C}.
    \end{cases}
\end{equation*}
Later, $|H_{3,1}(f)|\leq 8/9$ was calculated by Kwon et al.\cite{sharpstarlike}, which is recently best improved to $4/9$ for the functions in the class $\mathcal{S}^{*}$. Lecko et al. \cite{lecko 1/2 bound} obtained the sharp bound of $|H_{3,1}(f)|\leq 1/9$ for functions belonging to the class $\mathcal{S}^{*}(1/2)$. For more on Hankel determinants, refer \cite{sharp,kowal,rath,lecko 1/2 bound}.

Recently, Arora and Kumar\cite{kush} introduced a new subclass of $\mathcal{S}^{*}(\varphi)$ by choosing $\varphi(z)=1+\sinh^{-1}(z)$, given by
\begin{equation*}
\mathcal {S}^*_{\rho} =\Bigg\{f\in \mathcal {S} :\dfrac{z f'(z)}{f(z)}\prec 1+\sinh^{-1}(z)=:\rho(z)\Bigg\}.
\end{equation*}
The function $\rho(z)$ is a multivalued function having branch cuts along the line segments $(-\infty,-i)\cup (i,\infty)$ on the imaginary axis. Thus the function is analytic on $\mathbb{D}$. In geometric point of view, $\rho(z)$ takes the unit disk $\mathbb{D}$ onto a domain which is petal-shaped, $\Omega_{\rho}=\{\nu\in \mathbb{C}:|\sinh(\nu-1)|<1\}$, some coefficient problems of the functions belonging to the class $\mathcal{S}^{*}_{\rho}$ were deduced in \cite{4petalmgkhan}. In the present study, we determine the sharp bounds of first five coefficients and Hankel determinants for functions in $\mathcal{S}^{*}_{\rho}$. Further, the bounds of the sixth and seventh coefficient are obtained to estimate the bound of the fourth Hankel determinant.% combined with 
\section{Sharp Coefficient Problems for the Class $\mathcal{S}^{*}_{\rho}$}

\subsection{Preliminary}
Let $f\in \mathcal {S}^{*}_{\rho},$ then there exists a Schwarz function $w(z)=\sum_{k=1}^{\infty}w_{k}z^k$ so that 
\begin{equation}
	\dfrac{zf\;'(z)}{f(z)}= 1+\sinh^{-1} (w(z)).\label{4 schwarz}
	\end{equation}
Assume that $(p(z)+1)w(z)=(p(z)-1)$ with $p(z)=1+p_{1}z+p_{2}z^2+\cdots \in \mathcal {P}.$ After substituting the values of $w(z)$, $p(z)$ and $f(z)$ in (\refeq{4 schwarz}) and comparing the corresponding coefficients of either side of the equation, yields the following: 
	\begin{equation}
		a_2=\frac{1}{2}p_1, \quad a_3=\frac{1}{4}p_2, \quad a_4=\frac{1}{144}\bigg(-p_{1}^3-6p_{1}p_2+24p_3\bigg),\label{4 a2}
		\end{equation}
		\begin{equation}
		a_5=\frac{1}{1152}\bigg(5p_1^4-6p_1^2p_2-36p_2^2-48p_1p_3+144p_4\bigg), \label{4 a5}
	\end{equation}
\begin{equation}
    a_6=\frac{-54 p_1^5 + 355 p_1^3 p_2 + 150 p_1 p_2^2 - 1680 p_2 p_3 - 1080 p_1 p_4 + 
 2880 p_5}{28800},\label{4 a6}
\end{equation}
and \begin{align}
    a_7&=\frac{1}{2073600}\bigg(1031 p_1^6 - 17220 p_1^4 p_2 + 26100 p_1^2 p_2^2 + 9000 p_2^3 + 
 19200 p_1^3 p_3+ 33120 p_1 p_2 p_3\nonumber\\
 &\quad\quad\quad\quad\quad\quad - 57600 p_3^2 + 4320 p_1^2 p_4 -108000 p_2 p_4 - 69120 p_1 p_5\bigg).\label{4 a7}
\end{align}

To obtain the main result, we need the following:
\begin{lemma}\cite{rj}
Let $p\in \mathcal{P}$ be of the form $1+\sum_{n=1}^{\infty}p_nz^n.$ Then
\begin{equation}
        |p_1^4-3p_1^2p_2+p_2^2+2p_1p_3-p_4|\leq 2\label{p}
\end{equation}
and
\begin{equation}
       |p_3-2p_1p_2+p_1^3|\leq 2.\label{q}
    \end{equation}
    \end{lemma}
\begin{lemma}\cite{ma-minda}
Let $p\in \mathcal{P}$ be of the form $1+\sum_{n=1}^{\infty}p_nz^n.$ Then
\begin{align*}
      |p_2-\lambda p_1^2|\leq \begin{cases} 2-4\lambda,& \lambda\leq 0;\\
      2, & 0\leq \lambda\leq 1;\\
      4\lambda-2, & \lambda\geq 1
      \end{cases}
\end{align*}
when $\lambda<0$ or $\lambda>1,$ the equality holds if and only if $p(z)=(1+z)/(1-z)$ or one of its rotations. If $0<\lambda<1,$ then the inequality holds if and only if $p(z)=(1+z^2)/(1-z^2)$ or one of its rotations. If $\lambda=0,$ the equality holds if and only if $p(z)=(1+\eta)(1+z)/(2(1-z))+(1-\eta)(1-z)/(2(1+z))(0\leq \eta\leq 1)$ or one of its rotations. If $\lambda=1,$ the equality holds if and only if p is the reciprocal of one of the functions such that the equality holds in case of $\lambda=0.$ Though the above upper bound is sharp for $0<\lambda<1,$ still it can be improved as follows:
\begin{equation}
      |p_2-\lambda p_1^2|+\lambda|p_1|^2\leq 2 \quad (0<\lambda\leq 1/2)\label{use}
\end{equation}
and 
\begin{equation*}
      |p_2-\lambda p_1^2|+(1-\lambda)|p_1|^2\leq 2 \quad (1/2<\lambda\leq 1).
\end{equation*}
\end{lemma}
Along, with we recall that 
\begin{equation}
    \max_{0\leq t\leq 4}(At^2+Bt+C)=
    \begin{cases}
    C,& B\leq0, A\leq \frac{-B}{4};\\
    16A+4B+C, & B\geq0, A\geq \frac{-B}{8}\quad \text{or}\quad B\leq 0, A\geq \frac{-B}{4};\\
    \dfrac{4AC-B^2}{4A},& B>0, A\leq \frac{-B}{8}.\label{l}
    \end{cases}
\end{equation}
The following inequalities are required to obtain the bound of fourth order Hankel determinants result:
\begin{lemma}\label{2 pomi lemma}
Let $p(z)=1+\sum_{n=1}^{\infty}p_nz^n\in \mathcal{P}.$ Then
\begin{equation*}
    |p_n|\leq 2, \quad n\geq 1,\label{2 caratheodory1}
\end{equation*}
\begin{equation*}
    |p_{n+k}-\lambda p_n p_k|\leq \begin{cases}
    2, & 0\leq \lambda\leq 1;\\
    2|2\lambda-1|,& elsewhere
    \end{cases}\label{2 caratheodory2}
\end{equation*}
and \begin{equation*}
    |p_1^3-\lambda p_3|\leq
    \begin{cases}2|\lambda-4|,& \lambda\leq 4/3;\\ \\
    2\lambda\sqrt{\dfrac{\lambda}{\lambda-1}},& 4/3<\lambda.
    \end{cases}\label{2 caratheodory3}
\end{equation*}
\end{lemma}
\begin{lemma}\cite[Corollary 1]{alipcubelemma}\label{p1cubelemma}
Let $p(z)\in \mathcal{P}.$ Then
\begin{equation*}
    |\lambda p_1^3-(\lambda+1)p_1p_2+p_3|\leq \begin{cases}
    2,& 0\leq \lambda\leq 1;\\
    2|\lambda-1|,& elsewhere.
    \end{cases}
\end{equation*}
\end{lemma}

\subsection{Sharp Bound of the Fifth Coefficient}
The estimation of coefficient bounds of functions belonging to a subclass of $\mathcal{S}^{*}$ has always been a fundamental challenge in geometric function theory, they have a considerable impact on geometric properties, 
see \cite{goodman vol1,shelly}. %\cite[Page no. 17]{1 goodman}. 
For instance, the bound of the second coefficient determines the growth and distortion results for the class $\mathcal{S}$. We now proceed to obtain several coefficient bounds for functions in $\mathcal {S}^{*}_{\rho}$.

\begin{theorem}\label{4 sharpbounds}
Let $f(z)=z+\sum_{n=2}^{\infty}a_nz^n \in \mathcal {S}^{*}_{\rho},$ then 
		 $ |a_2|\leq 1,
		  |a_3|\leq 1/2,
		|a_4|\leq 1/3,$ and
		$|a_5|\leq 907/1632\approx 0.55576.$
	
All these bounds are sharp.
\end{theorem}
\begin{proof} 
For finding the bounds of $a_2$ as in (\ref{4 a2}), we utilize the already known coefficient estimate for the Carath\'{eodory} class of functions $|p_n|\leq 2$ for $(n\geq 1).$ Thus, $|a_2|\leq 1$.\\
The bound of $a_3$ is $1/2$ as per \cite{4petalmgkhan}.\\
For $a_4,$ equation (\ref{4 schwarz}) is re-written as:
\begin{equation}
		zf\;'(z)=(1+\sinh^{-1} (w(z)))f(z). \label{4 re}
\end{equation}
On substituting $f(z)$ from (\ref{form}) and $w(z)=\sum_{k=1}^{\infty}w_{k}z^k$ in equation (\ref{4 re}) and comparing the coefficients of $z^4$, we get $$3a_4=w_3+\frac{1}{3}w_1^3+\frac{3}{2}{w_1}w_2.$$
Using \cite[Lemma 2]{poko}, we obtain that $|a_4|\leq 1/3.$\\
Now, from (\ref{4 a5}), we have 
\begin{align*}
	a_5&=\frac{-1}{8}\bigg[\frac{-5}{144}{p_{1}}^4+\frac{1}{4}{p_{2}}^2+\frac{1}{3}p_{1}p_{3}+\frac{1}{24}{p_{1}^2}p_2-p_4\bigg]\\
	&=\dfrac{-1}{8}\bigg[\dfrac{1}{4}X-\dfrac{1}{6}p_1Y+\dfrac{11}{24}p_1^2Z-\dfrac{3}{4}p_4\bigg]\\
	&\leq \dfrac{1}{8}\bigg[\dfrac{1}{4}|X|+\dfrac{1}{3}|Y|+\dfrac{7}{12}|p_1|^2|Z|+\dfrac{3}{2}\bigg|,
\end{align*}
with $X=p_1^4+p_2^2-3p_1^2p_2+2p_1p_3-p_4$, $Y=p_1^3-2p_1p_2+p_3$ and $Z=p_2-(17/66)p_1^2$. Now from equations (\ref{p})-(\ref{use}), we have $|X|\leq 2$, $|Y|\leq 2$ and $|Z|\leq 2$ respectively. Upon substituting these values in the expression of $a_5$ given above, we have
\begin{equation*}
		    |a_5|\leq \dfrac{1}{8}\bigg(\dfrac{8}{3}+\dfrac{11}{12}|p_1|^2-\dfrac{17}{144}|p_1|^4\bigg).
\end{equation*}
Now, we obtain $|11|p_1|^2/12-17|p_1|^4/144|\leq 121/68$ using the equation (\ref{l}) by taking $A=-17/144$, $B=11/12$ and $C=0,$ eventually estimates the desired bound for $a_5.$  
By taking, $p_1=p_4=2$, $p_2=-0.333333 + 1.45521i$ and $p_3=-2$, we prove the sharpness of the result. \\
For the initial coefficients $a_n(n=2,3,4)$, the following function acts as the extremal function:
	$$f_n(z)=z\exp\bigg(\int_{0}^{z}\frac{\sinh^{-1} (t^{n-1})}{t}dt\bigg).$$
\end{proof}

\begin{remark}
The bound of $a_3$ obtained in \cite{4petalmgkhan} is not sharp. So, in Theorem \ref{4 sharpbounds}, we have provided its corresponding extremal function.
\end{remark}

\subsection{Sharp Hankel Determinants}

The estimation of the Hankel determinant bounds for functions in $\mathcal{S}^{*}$ and its various subclasses in terms of $q$ and $n$, is trending these days. 
The well-known formula for $p_2$, $p_3$ and $p_4$ \cite{rj,lemma1}, listed below in Lemma \ref{pformula}, plays an important role in finding the best possible bounds of Hankel determinants, which serves as the foundation for establishing our main results.

\begin{lemma}\label{pformula}
Let $p\in \mathcal {P}$ of the form $1+\sum_{n=1}^{\infty}p_n z^n.$ Then
	\begin{equation}
		2 p_2=p_1^2+\gamma (4-p_1^2),\label{b2}
	\end{equation}
\begin{equation}
	4p_3=p_1^3+2p_1(4-p_1^2)\gamma -p_1(4-p_1^2) {\gamma}^2+2(4-p_1^2)(1-|\gamma|^2)\eta, \label{b3}
\end{equation}
and \begin{equation}
	8p_4=p_1^4+(4-p_1^2)\gamma (p_1^2({\gamma}^2-3\gamma+3)+4\gamma)-4(4-p_1^2)(1-|\gamma|^2)(p_1(\gamma-1)\eta+\bar{\gamma}{\eta}^2-(1-|\eta|^2)\rho), \label{b4}
\end{equation}
for some $\gamma$, $\eta$ and $\rho$ such that $|\gamma|\leq 1$,  $|\eta|\leq 1$ and $|\rho|\leq 1.$
\end{lemma}
Now, we will determine the second-order Hankel determinant $H_{2,2}(f)$ for $f\in \mathcal{S}^{*}_{\rho}.$
\begin{theorem}
Let $f$ belongs to $\mathcal{S}_{\rho}^{*}$. Then
\begin{equation}
    |H_{2,2}(f)|\leq \frac{1}{4}.\label{4 h22}
\end{equation}
This result is the best possible.
\end{theorem}
\begin{proof}
On  substituting the values of $a_i$ for $i=2,3,4$ from (\ref{4 a2}) in (\ref{h22}), we obtain 
\begin{equation*}
    a_2a_4-a_3^2=-\frac{p_1^4}{288}-\frac{p_1^2p_2}{48}-\frac{p_2^2}{16}+\frac{p_1p_3}{12}.
\end{equation*}
We express $p_i$ for $i=2,3$ in the form of $p_1$ through equations (\ref{b2}) and (\ref{b3}). Also, without loss of generality, consider $p_1:=p$ where $0\leq p\leq 2$, we have
\begin{equation*}
    |a_2a_4-a_3^2|=\bigg|-\frac{5}{576}p^4-\frac{1}{64}\gamma^2 (4-p^2)-\frac{1}{48}p^2 \gamma^2 (4-p^2)+\frac{1}{24}p (4-p^2)(1-|\gamma|^2)\eta\bigg|.
\end{equation*}
On utilizing the triangle inequality and replacing $|\eta|\leq 1$, $|\gamma|=d$, with $d\leq 1.$ We arrive at 
\begin{align*}
    |a_2a_4-a_3^2|&=\frac{5}{576}p^4+\frac{1}{64}d^2 (4-p^2)+\frac{1}{48}p^2 d^2 (4-p^2)+\frac{1}{24}p (4-p^2)(1-d^2)\\
    &=:\psi(p,d).
\end{align*}
It is a simple exercise to show that $\psi'(p,d)\geq 0$ on $[0,1]$, so that $\psi(p,d)\leq \psi(p,1).$ On substituting $d=1$, we get
\begin{align*}
    |a_2a_4-a_3^2|&=\frac{5}{576}p^4+\frac{1}{64} (4-p^2)+\frac{1}{48}p^2 (4-p^2)\\
    &=:\psi(p,1).
\end{align*}
Also, $\psi'(p,1)<0$, and so $\psi(p,1)$ is a decreasing function. Thus, the maximum is obtained at $p=0$, given as
\begin{equation*}
    |H_{2,2}(f)|=|a_2a_4-a_3^2|\leq \frac{1}{4}.
\end{equation*}
The corresponding extremal function for equation (\ref{4 h22}) is $f:\mathbb{D}\rightarrow \mathbb{C}$ defined as
\begin{equation}
f(z)=z\exp\bigg(\int_{0}^{z}\dfrac{\sinh^{-1}(t^2)}{t}dt\bigg)=z+\dfrac{z^3}{2}+\dfrac{z^5}{8}-\cdots,\label{4 extre fkete}
\end{equation}
with $f(0)=0$ and $f'(0)=1$. It acts as an extremal function for the bounds of $|a_2a_4-a_3^2|$ for the values of $a_2=0$ and $a_3=1/2.$ 
\end{proof}

\begin{remark}\label{4 remark} \cite{4petalmgkhan} 
Let $f(z)=z^p+\sum_{n=p+1}^{\infty}a_nz^n\in \mathcal{S}^{*}_{\rho}$. Then,
\begin{equation}
    |a_{p+2}-\lambda a_{p+1}^2|\leq \frac{p}{2}\max \{1,|2\lambda-1|\},\label{4 fkete}
\end{equation}
and \begin{equation*}
|a_{p+2}|\leq \frac{p}{2}.
\end{equation*}
\end{remark}
On taking $p=\lambda=1$ in (\ref{4 fkete}), authors in \cite{4petalmgkhan} obtained the non-sharp bound given as
\begin{equation*}
    |a_3-a_2^2|\leq \frac{1}{2}.
\end{equation*}
 
The function defined in equation (\ref{4 extre fkete}) guarantees the sharpness of the upper-bound of $|a_3-a_2^2|$ for the values of $a_2=0$ and $a_3=1/2.$ \\ \\
Finding the bounds of higher order Hankel determinants for several subclasses of $\mathcal{S}^{*}(\varphi)$ is comparatively a difficult task and has recently become a matter of great interest to many researchers, see \cite{sharp,kowal,rath,lecko 1/2 bound}. So, by overcoming the difficulties, we present the sharp bound of $H_{3,1}(f)$ for functions related to the class $\mathcal{S}^{*}_{\rho}$.

\begin{theorem}\label{4 Ph31}
Let $f\in \mathcal {S}^{*}_{\rho}.$ Then 
\begin{equation}
	|H_{3,1}(f)|\leq \dfrac{1}{9}.\label{4 9.5}
\end{equation}
The result is sharp.
\end{theorem}
\begin{proof}  As the class $\mathcal {P}$ is rotationally invariant and $0\leq p_1\leq 2$. Suppose $p:=p_1$, then replacing the values of $a_i's(i=2,3,4,5)$ from equations (\ref{4 a2}) and (\ref{4 a5}) in the equation (\ref{1h3}), we get,
\begin{align*}H_{3,1}(f)&=\dfrac{1}{41472}\bigg(-47p^6+3p^4p_2-234p^2{p_2}^2-972{p_2}^3+528p^3p_3+1872p p_2 p_3-1152p_3^2-1296p^2p_4\\
&\quad\quad\quad\quad\quad+1296p_2p_4\bigg).
\end{align*}
	On simplification using (\ref{b2})-(\ref{b4}), we obtain
	
	$$H_{3,1}(f)=\dfrac{1}{82944}\bigg(\Gamma_1(p,\gamma)+\Gamma_2(p,\gamma)\eta+\Gamma_3(p,\gamma){\eta}^2+\phi(p,\gamma,\eta)\rho\bigg),$$
	where $\gamma,\eta,\rho\in \mathbb {D},$
\begin{align*}	\Gamma_1(p,\gamma):&=-25p^6-135{\gamma}^3p^2(4-p^2)^2+18{\gamma}^4p^2(4-p^2)^2-324{\gamma}^3(4-p^2)^2+72{\gamma}p^4(4-p^2)\\
&\quad-648{\gamma}^2p^2(4-p^2)+42p^4{\gamma}^2(4-p^2)-162p^4{\gamma}^3(4-p^2),\\
	\Gamma_2(p,\gamma):&=24(1-|\gamma|^2)(4-p^2)(10p^3+27{\gamma}p^3+(4-p^2)(18{\gamma}p-3p\gamma^2)),\\
	\Gamma_3(p,\gamma):&=72(1-|\gamma|^2)(4-p^2)(-8(4-p^2)-|\gamma|^2(4-p^2)+9p^2\bar{\gamma}),\\
\phi(p,\gamma,\eta):&=648(1-|\gamma|^2)(4-p^2)(1-|\eta|^2)(-p^2+\gamma(4-p^2)).
\end{align*}
Assuming $x=|\gamma|$, $y=|\eta|$ and employing $|\rho|\leq 1,$ we get
\begin{align*}
	|H_{3,1}(f)|\leq \dfrac{1}{82944}\bigg(|\Gamma_1(p,\gamma)|+|\Gamma_2(p,\gamma)|y+|\Gamma_3(p,\gamma)|y^2+|\phi(p,\gamma,\eta)|\bigg)\leq M(p,x,y),
\end{align*}
where 
\begin{equation}
		M(p,x,y)=\dfrac{1}{82944}\bigg(m_1(p,x)+m_2(p,x)y+m_3(p,x)y^2+m_4(p,x)(1-y^2)\bigg)\label{new}
\end{equation} 
with 
\begin{align*}
	m_1(p,x):&=25p^6+135x^3p^2(4-p^2)^2+18x^4p^2(4-p^2)^2+324x^3(4-p^2)^2+72xp^4(4-p^2)\\
	&\quad +648x^2p^2(4-p^2)+42x^2p^4(4-p^2)+162x^3p^4(4-p^2),\\
m_2(p,x):&=24(1-x^2)(4-p^2)(10p^3+27p^3x+(4-p^2)(18px+3px^2)),\\	m_3(p,x):&=72(1-x^2)(4-p^2)((8+x^2)(4-p^2)+9p^2x),\\
m_4(p,x):&=648(1-x^2)(4-p^2)(p^2+x(4-p^2)).
	\end{align*}
In the closed cuboid $U:[0,2]\times [0,1]\times [0,1]$, we must now maximise $M(p,x,y)$, by identifying the maximum values in the interior, on the twelve edges and inside $U$.
\begin{enumerate}
\item We start by taking into account, every internal point of $U$. Assume that $(p,x,y)\in (0,2)\times (0,1)\times (0,1).$ We differentiate equation (\ref{new}) partially with respect to $y$ to identify the points at which maximum is attained inside $U$. We get
\begin{align*}
	    \dfrac{\partial M}{\partial y}&=\dfrac{(4 - p^2)}{3456} (1 - x^2) \bigg(12 p x (6 + x) + p^3 (10 + 9 x - 3 x^2)+24p^2y(8-9x+x^2)\\
	    & \quad \quad \quad \quad\quad \quad\quad \quad \quad  -6 p^2y (17 - 18 x + x^2)\bigg). 
\end{align*}
Now $\dfrac{\partial M}{\partial y}=0$ gives
\begin{equation*}
	    y=y_0:=\dfrac{10 p^3 + 72 p x + 9 p^3 x + 12 p x^2 - 3 p^3 x^2}{6 (x-1) (32 - 17 p^2 - 4 x + p^2 x)}.
\end{equation*}
The existence of the critical points requires that $y_0$ should belong to $(0,1),$ which is only possible when 
\begin{equation}
	    p^3 (10 + 9 x - 3 x^2) + 24 (8 - 9 x + x^2)-6 p^2 x^2 + 12 p x (6 + x)<6 p^2 (17 - 18 x).\label{4 h1}
\end{equation}
%and
%\begin{equation}
%	    4(9x-8)+p^2(17+8x)<x^2(4-p^2).\label{h2}
%\end{equation}
Next, we determine the values satisfying the inequality (\ref{4 h1})
%and (\ref{h2}) 
for the existence of critical points using hit and trial method. If we assume $p$ tends to 2, then equation (\ref{4 h1}) holds for all $x<17/54$ but there does not exist any $x\in (17/54,1)$ satisfying the equation (\ref{4 h1}). When $p$ tends to 0, there does not exist any $x\in (0,1)$ satisfying equation (\ref{4 h1}). Similarly, if we assume $x$ tends to 0, then equation (\ref{4 h1}) holds for all $p>1.48422$. On calculations, we observe that there does not exist any $x\in (0,1)$ when $p\in (0,1.48422).$ If $x$ tends to 1, then no $p\in (0,2)$ satisfies equation (\ref{4 h1}). %Similarly, if we assume $p$ tends to 0 and 2, then there is no $x\in (0,1)$ satisfying equation (\ref{h2}). Also, when $x$ tends to 0, then equation (\ref{h2}) holds for $p< 1.37199$ and no $x\in (0,1)$ exists for $p\in (1.37199,2).$
Thus, the domain for the solution of equation (\ref{4 h1}) is $(1.48422,2)\times (0,17/54).$ Further calculations depicts that $\frac{\partial M}{\partial p}|_{y=y_0}\neq 0$ in the domain, $(1.48422,2)\times (0,17/54).$
Hence, we conclude that the function $M$ has no critical point in $(0,2)\times (0,1)\times (0,1).$ 
\item The interior of each of the cuboid $U$'s six faces is now being considered.\\
For $p=0$,
\begin{equation}
	        n_1(x,y):=\dfrac{(1 - x^2) ((8 + x^2) y^2 + 9 x (1 - y^2))}{72},\label{4 9.4}
\end{equation}
with $x,y\in (0,1)$. We observe that $(0,1)\times(0,1)$ does not contain any critical point of $n_1$ as 
\begin{equation*}
    \dfrac{\partial n_1}{\partial y}=\dfrac{y(x-8)(1-x^2)(x-1)}{36}\neq 0,\quad x,y\in (0,1).
\end{equation*}
For $p=2$,
\begin{equation}
    M(2,x,y):=\dfrac{25}{1296},\quad x,y\in (0,1).\label{4 9.3}
\end{equation}
For $x=0$, 
\begin{equation}
    n_2(p,y):=\dfrac{25 p^6 + (4 - p^2)(240 p^3  y + 576 (4 - p^2) y^2 + 
 648 p^2  (1 - y^2))}{82944}\label{4 9.1}
\end{equation}
for $p\in (0,2)$ and $y\in (0,1).$ To determine the points of maxima, solve $\partial n_2/\partial p$ and $\partial n_2/\partial y$ to examine the values of maxima. After resolving $\frac{\partial n_2}{\partial y}=0,$ we get
\begin{equation}
    y=-\dfrac{5p^3}{3(32-17p^2)}(=:y_0).\label{4 y}
\end{equation}
In order to have $y_0\in (0,1),$ for the given range of $y$, $p>p_0\approx 1.48422$ is needed. Based on the calculations, $\frac{\partial n_2}{\partial p}=0$ gives
\begin{equation}
    5184 p - 2592 p^3 + 150 p^5 + 2880 p^2 y - 1200 p^4 y - 14400 p y^2 + 
 4896 p^3 y^2=0.\label{4 9}
\end{equation}
On substituting equation (\ref{4 y}) in equation (\ref{4 9}), we get
\begin{equation}
    294912 p -460800 p^3 + 239904 p^5 - 44816 p^7 + 1275 p^9=0.\label{4 40}
\end{equation}
Based on calculations, the solution of (\ref{4 40}) is $p\approx 1.20623$ in the interval $(0,2)$. So, in $(0,2)\times (0,1)$, $n_2$ does not have any critical point.\\
For $x=1$, 
\begin{equation}
    n_3(p,y):=M(p,1,y)=\dfrac{(2592+1224p^2-222p^4-49p^6)}{41472}, \quad p\in (0,2),\label{4 9.2}
\end{equation}
and $p=p_0:\approx1.32162$ comes out to be the critical point at which $n_3$ achieves its maximum, approximately $0.0914236$, while computing $\partial n_3/\partial p=0.$\\ 
For $y=0$,
\begin{align*}
    n_4(p,x):&=\dfrac{1}{82944}\bigg(p^6 (25 - 42 x^2 - 27 x^3 + 18 x^4)-288 (-36 x + 18 x^3 - xp^4)\\
    &\quad \quad \quad \quad\quad  -12 p^4 (54 - 54 x - 14 x^2 + 63 x^3 + 12 x^4)\\
    &\quad \quad\quad \quad\quad+72 p^2 (36 - 72 x + 66 x^3 + 4 x^4 - xp^4)\bigg).
\end{align*}
On computations,
\begin{align*}
    \dfrac{\partial n_4}{\partial x}&=\dfrac{1}{82944}\bigg( 12 p^4 (54 + 28 x - 189 x^2 - 48 x^3) + p^6 (-84 x - 81 x^2 + 72 x^3) \\
    &\quad \quad\quad \quad\quad -288 (-36 + 54 x^2) + 72 p^2 (-72 + 198 x^2 + 16 x^3)\bigg)
\end{align*}
and \begin{align*}
    \dfrac{\partial n_4}{\partial p}&=\dfrac{p}{13824}\bigg(p^4 (25 - 42 x^2 - 27 x^3 + 18 x^4)-8 p^2 (54 - 54 x - 14 x^2 + 63 x^3 + 12 x^4) \\&\quad\quad\quad\quad \quad+ 
 24  (36 - 72 x + 66 x^3 + 4 x^4 - xp^4) \bigg).
\end{align*}
in $(0,2)\times (0,1)$, no solution exists for equations, $\partial n_4/\partial x=0$ and $\partial n_4/\partial p=0$, as per the calculations.\\
For $y=1$, 
\begin{align*}
    n_5(p,x):&=\frac{1}{82944}\bigg(25 p^6 +(4 - p^2)( 72 p^4  x + 648 p^2 x^2 +42 p^4 x^2+162 p^4 x^3  \\
& \quad\quad\quad\quad\quad+18 p^2 (4 - p^2) x^4 +72 (1 - x^2) (9 p^2 x+ (4 - p^2) (8 + x^2)) \\
 &\quad\quad\quad\quad\quad+ 135 p^2 (4 - p^2) x^3 + 324 (4 - p^2) x^3 +288px (1 - x^2) (6 + x) \\
 &\quad\quad\quad\quad\quad+24p^3(1-x^2) (10 + 9 x - 3 x^2))\bigg).
\end{align*}
The two equations $\partial n_5/\partial x=0$ and $\partial n_5/\partial p=0$ do not assume any solution in $(0,2)\times (0,1).$

\item Now, the edges of the cuboid $U$ are taken into consideration to identify the maximum possible values achieved by $M(p,x,y)$. From equation (\ref{4 9.1}), we have $M(p,0,0)=r_1(p):=(25p^6+2592p^2-648p^4)/82944.$ We observe that $d r_1/dp=0$ for $p=\delta_0:=0$ and $p=\delta_1:=1.51933$ in $[0,2]$ as points of minima and maxima respectively. %The maximum value of $r_1(p)$ is $\approx 0.0342145.$
Hence, 
\begin{equation*}
    M(p,0,0)\leq 0.0342145.
\end{equation*}
Now considering the equation (\ref{4 9.1}) at $y=1,$ we get $M(p,0,1)=r_2(p):=(9216 - 4608 p^2 + 3168 p^4 - 623 p^6)/82944.$ It is easy to observe that $d r_2/dp$ decreases in $[0,2]$ and maximum value is achieved at the point $p=0$. So,
\begin{equation*}
     M(p,0,1)\leq \dfrac{1}{9}, \quad p\in [0,2].
    \end{equation*}
Through computations, equation (\ref{4 9.1}) shows that $y=1$ is the point at which $M(0,0,y)$ assumes its maximum value. This implies that
\begin{equation*}
    M(0,0,y)\leq \dfrac{1}{9}, \quad y\in [0,1].
\end{equation*}
As the expression in (\ref{4 9.2}) does not involve $x$, which implies $M(p,1,1)=M(p,1,0)=r_3(p):=(2592+1224p^2-222p^4-49p^6)/41472.$ Now, $r_3'(p)=2448 p - 888 p^3 - 294 p^5=0$ when $p=\delta_2:=0$ and $p=\delta_3:=1.32162$ in the interval $[0,2]$ with $\delta_2$ and $\delta_3$ as points of minima and maxima respectively. Therefore
\begin{equation*}
    M(p,1,1)=M(p,1,0)\leq 0.0914236,\quad p\in [0,2].
\end{equation*}
Equation (\ref{4 9.2}) reduces to $M(0,1,y)=1/16$ at $p=0$. The equation (\ref{4 9.3}) is completely free from $p$, $x$ and $y$. Thus the maximum value of $M(p,x,y)$ on the edges $p=2, x=0; p=2, x=1; p=2, y=0$ and $p=2, y=1,$ respectively, is
\begin{equation*}
    M(2,0,y)=M(2,1,y)=M(2,x,0)=\dfrac{25}{1296},\quad x,y\in [0,1].
\end{equation*}
Equation (\ref{4 9.1}) implies that, $M(0,0,y)=y^2/9.$ Further, it implies that
\begin{equation*}
    M(0,0,y)\leq \dfrac{1}{9},\quad y\in [0,1].
    \end{equation*}
Through equation (\ref{4 9.4}), we obtain $M(0,x,1)=r_4(x):=(8-7x^2-x^4)/72.$ Upon calculations, we see that, $r_4$ is decreasing in $[0,1]$ and obtain its maximum at $x=0$. Hence
 \begin{equation*}
     M(0,x,1)\leq \dfrac{1}{9},\quad x\in [0,1].
 \end{equation*}
From equation (\ref{4 9.4}), we get $M(0,x,0)=r_5(x):=x(1-x^2)/8.$ On further calculations, we get $r_5'(x)=0$ at $x=1/\sqrt{3}.$ Also, $r_5(x)$ increases in $[0,1/\sqrt{3})$ and decreases in $(1/\sqrt{3},1].$ So, $1/\sqrt{3}$ is the point of maxima. Thus
 \begin{equation*}
     M(0,x,0)\leq 0.0481125,\quad x\in [0,1].
 \end{equation*}
\end{enumerate}
The inequality (\ref{4 9.5}) holds in all the cases as discussed above. The function $f:\mathbb{D}\rightarrow \mathbb{C}$ defined as
\begin{equation*}
f(z)=z\exp\bigg(\int_{0}^{z}\dfrac{\sinh^{-1}(t^3)}{t}dt\bigg)=z+\dfrac{z^4}{3}+\dfrac{z^7}{18}+\cdots,\label{4 extremal}
\end{equation*}
with $f(0)=0$ and $f'(0)=1$,
%The function defined in equation (\ref{4 extremal})
proves the sharpness for the bounds of $|H_3(1)|$ for the values of $a_2=a_3=a_5=0$ and $a_4=1/3.$
\end{proof}
Following the technique as in Theorem \ref{4 Ph31}, we conclude the bound of $H_{2,3}:$
\begin{corollary}
For $f\in \mathcal{S}^{*}_{\rho},$
\begin{equation*}
    |H_{2,3}(f)|\leq 0.146048.
\end{equation*}
\end{corollary}

\section{Fourth Hankel Determinant}
The expression for the fourth order Hankel determinant for the particular choices of $q$ and $n$ is as follows:
\begin{equation}
    H_{4,1}(f)=a_7H_3(1)-a_6\Omega_1+a_5\Omega_2-a_4\Omega_3\label{4 4hankelexpression}
\end{equation}
where 
\begin{equation}
    \Omega_1:=a_3a_6-a_4a_5-a_2(a_2a_6-a_3a_5)+a_4(a_2a_4-a_3^2),\label{4 omega1}
\end{equation}
\begin{equation}
    \Omega_2:=(a_4a_6-a_5^2)-a_2(a_3a_6-a_4a_5)+a_3(a_3a_5-a_4^2),\label{4 omega2}
\end{equation}	
and \begin{equation}
    \Omega_3:=a_2(a_4a_6-a_5^2)-a_3(a_3a_6-a_4a_5)+a_4(a_3a_5-a_4^2)\label{4 omega3}.
\end{equation}	
Now, we try to estimate the values of $\Omega_1$, $\Omega_2$, and $\Omega_3$.
By using the values of $a_i$'s as a substitute in equation (\ref{4 omega1}), we get
\begin{align*}
 460800\Omega_1&=241 p_1^7 - 1186 p_1^5 p_2+ 720 p_1^2 p_2 p_3 - 11520 p_1^2 p_5+11520 p_2 p_5+ 5280 p_1 p_2 p_4 - 600 p_1 p_2^3\\
 &\quad+ 920 p_1^3 p_2^2 - 9120 p_2^2 p_3 +4720 p_1^3 p_4- 9600 p_3 p_4+ 9600 p_1 p_3^2-1000 p_1^4 p_3 .   
\end{align*}
By applying triangle inequality and Lemma \ref{2 pomi lemma} in the above-mentioned equation, we reach at the following inequality
\begin{align*}
    460800|\Omega_1|&\leq |p_1^5(241 p_1^2 - 1186  p_2)|+ |p_1^2(720  p_2 p_3 - 11520 p_5)|+|p_2(11520  p_5+ 5280 p_1  p_4)|+ 600 |p_1 p_2^3|\nonumber \\
 &\quad+|p_2^2( 920 p_1^3  - 9120  p_3)| +|p_4(4720 p_1^3 - 9600 p_3)|+|p_1 p_3( 9600 p_3  -1000 p_1^3)| \nonumber\\
 &\leq 265984+32000\sqrt{\frac{3}{43}}+37760\sqrt{\frac{30}{61}}+29184\sqrt{\frac{285}{41}}.
\end{align*}
Thus \begin{equation}
    |\Omega_1|\leq \frac{265984+32000\sqrt{\frac{3}{43}}+37760\sqrt{\frac{30}{61}}+29184\sqrt{\frac{285}{41}}}{460800}\approx 0.820011.\label{4 omega1value}
\end{equation}

On similar pattern
\begin{align*}
 460800\Omega_2&=-77 p_1^8 + 692 p_1^6 p_2+ 11200 p_1^3 p_2 p_3- 61440 p_2 p_3^2 +23040 p_1 p_2^2 p_3 - 5540 p_1^4 p_2^2\\
 &\quad+ 61440 p_3 p_5- 2560 p_1^3 p_5+ 53760 p_1 p_3 p_4 - 4640 p_1^4 p_4 + 2048 p_1^5 p_3- 19200 p_1^2 p_3^2\\
 &\quad+ 18240 p_1^2 p_2 p_4-61440 p_1 p_2 p_5+57600 p_2^2 p_4- 57600 p_4^2  - 4800 p_1^2 p_2^3 - 10800 p_2^4.   
\end{align*}
Lemma  \ref{2 pomi lemma} and triangle inequality leads us to
\begin{align*}
 3686400|\Omega_2|&\leq |p_1^6(-77 p_1^2 + 692  p_2)|+  |p_2 p_3(11200p_1^3 - 61440 p_3)|+|p_1 p_2^2(23040  p_3- 5540 p_1^3)|\\
 &\quad+|p_5( 61440 p_3  - 2560 p_1^3 )|+|p_1p_4( 53760 p_3 - 4640 p_1^3)| +|p_1^2p_3( 2048 p_1^3- 19200 p_3)|\\
 &\quad+|p_1 p_2 ( 18240 p_1 p_4-61440 p_5)|+| p_4(57600 p_2^2- 57600 p_4)|+|p_2^3(-4800 p_1^2-10800 p_2)|\\
 &\leq 1136896 + 716800 \sqrt{\frac{3}{157}} + 81920 \sqrt{\frac{3}{67}} + 212736 \sqrt{\frac{2}{35}} + 
 74240 \sqrt{\frac{21}{307}}+ 10240\sqrt{\frac{6}{23}}.
\end{align*}
Thus \begin{align}
    |\Omega_2|&\leq \frac{1136896 + 716800 \sqrt{\frac{3}{157}} + 81920 \sqrt{\frac{3}{67}} + 212736 \sqrt{\frac{2}{35}} + 
 74240 \sqrt{\frac{21}{307}}+ 10240\sqrt{\frac{6}{23}}}{3686400}\nonumber\\
 &\quad\approx 0.360465.\label{4 omega2value}
\end{align}

Again, we have
\begin{align*}
 597196800\Omega_3&=-1537 p_1^9 + 5868 p_1^7 p_2+ 
 4976640 p_1 p_3 p_5 -207360 p_1^4 p_5- 42876 p_1^5 p_2^2  - 
 51840 p_1^2 p_2^2 p_3\\
 &\quad + 1244160 p_1^2 p_3 p_4-246240 p_1^5 p_4+ 622080 p_2^3 p_3- 449280 p_1^3 p_2^3+ 288 p_1^6 p_3- 172800 p_1^3 p_3^2 \\
 &\quad+ 
 596160 p_1^3 p_2 p_4+6220800 p_2 p_3 p_4+ 734400 p_1^4 p_2 p_3 - 2903040 p_1 p_2 p_3^2+ 2177280 p_1 p_2^2 p_4  \\
 &\quad-97200 p_1 p_2^4- 3732480 p_2^2 p_5- 1244160 p_1^2 p_2 p_5- 2764800 p_3^3 - 4665600 p_1 p_4^2  .   
\end{align*}
By applying Lemma \ref{2 pomi lemma} along with triangle inequality in the above equation, we arrive at the following:
\begin{align*}
597196800|\Omega_3|&\leq |p_1^7(-1537 p_1^2 + 5868  p_2)|+|p_1p_5(4976640 p_3-207360 p_1^3)|+|p_1^2 p_2^2(-42876 p_1^3-51840 p_3)|\\
&\quad +|p_1^2p_4(1244160p_3-246240 p_1^3)|+ |p_2^3(622080p_3-449280p_1^3)|+|p_1^3p_3(288 p_1^3 - 172800 p_3)| \\
 &\quad+| p_2 p_4(596160 p_1^3+6220800 p_3)|+|p_1 p_2 p_3(734400 p_1^3 - 2903040p_3)|+ 4665600 |p_1 p_4^2|\\
 &\quad+ |p_1p_2^2(2177280p_4  -97200p_2^2)|+|p_2p_5 (-3732480 p_2- 1244160p_1^2)|+ 2764800 |p_3^3|\\
 &\leq 221394816 + 55296000  \sqrt{\frac{6}{599}}+ 79626240 \sqrt{\frac{6}{77}}+ 
 185794560 \sqrt{\frac{21}{251}}+ 79626240 \sqrt{\frac{6}{23}}\\
 &\quad+ 5971968 \sqrt{10}.
\end{align*}
Thus \begin{align}
    |\Omega_3|&\leq \frac{ 221394816 + 55296000  \sqrt{\frac{6}{599}}+ 79626240 \sqrt{\frac{6}{77}}+185794560 \sqrt{\frac{21}{251}}+ 79626240 \sqrt{\frac{6}{23}}+ 5971968 \sqrt{10}}{597196800}\nonumber\\
 &\quad\approx 0.606922.\label{4 omega3value}
\end{align}
Next, we obtain the bounds of the sixth and seventh coefficient of functions related to the class $\mathcal{S}^{*}_{\rho}$, as they play a significant role in obtaining the bound of $H_{4,1}$, given as follows:
\begin{lemma}\label{4 a6a7bound}
Let $f(z)=z+\sum_{n=2}^{\infty}a_nz^n \in \mathcal {S}^{*}_{\rho},$ then $|a_6|\leq146/225\approx 0.648889$ and $|a_7|\leq 1031/1080\approx 0.95463.$
\end{lemma}	

\begin{proof}
To find the bound of $|a_6|,$ we rearrange the terms of the equation (\ref{4 a6}) and use the triangle inequality such that:
\begin{align}
    28800|a_6|&=|-54 p_1^5 + 355 p_1^3 p_2 + 150 p_1 p_2^2 - 1680 p_2 p_3 - 1080 p_1 p_4 +
 2880 p_5|\nonumber\\
 &=\bigg|2880\bigg(p_5-\frac{1680}{2880}p_2p_3\bigg)+150p_1p_2\bigg(p_2+\frac{355}{150}p_1^2\bigg)-54 p_1^5 - 1080 p_1 p_4\bigg|\nonumber\\
 &\leq 2880\bigg|p_5-\frac{1680}{2880}p_2p_3\bigg|+150\bigg|p_1p_2\bigg(p_2+\frac{355}{150}p_1^2\bigg)\bigg|+54|p_1^5|+1080|p_1 p_4|.\label{4 lemmaa6}
\end{align}
Using Lemma \ref{2 pomi lemma} in equation (\ref{4 lemmaa6}),
\begin{equation}
    2880\bigg|p_5-\frac{1680}{2880}p_2p_3\bigg|\leq 5760, \quad 150\bigg|p_1p_2\bigg(p_2+\frac{355}{150}p_1^2\bigg)\bigg|\leq 6880,\quad  54|p_1^5|\leq 1728\label{4 a6cal1}
\end{equation}
and \begin{equation}
   1080|p_1 p_4|\leq 4320.\label{4 a6cal2}
\end{equation}
From equation (\ref{4 lemmaa6})-(\ref{4 a6cal2}), we obtain
%(\ref{4 a6cal1}) and (\ref{4 a6cal2}), we obtain
\begin{equation*}
    |a_6|\leq \frac{18688}{2880}\approx 0.648889.
\end{equation*}
Thus we get the desired estimate of $|a_6|.$\\
Next, we suitably rearrange the terms given in equation (\ref{4 a7}) as
\begin{align*}
    2073600 a_7&=p_1^3( 1031 p_1^3 - 17220 p_1 p_2+19200p_3)+p_2^2( 9000 p_2+26100p_1^2)   \\
 &\quad+p_1(33120p_2p_3-69120p_5)
  - 57600 p_3^2-108000 p_2 p_4
\end{align*}
or
\begin{align}
   2073600 |a_7|&\leq |p_1^3( 1031 p_1^3 - 17220 p_1 p_2+19200p_3)|+|p_2^2( 9000 p_2+26100p_1^2)| \nonumber  \\
 &\quad+|p_1(33120p_2p_3-69120p_5)|
  + 57600 |p_3^2|+108000 |p_2 p_4|. \label{4 a7cal}
\end{align}
Now, using Lemma \ref{2 pomi lemma} and Lemma \ref{p1cubelemma} in equation (\ref{4 a7cal}), we obtain
\begin{equation}
|p_1^3( 1031 p_1^3 - 17220 p_1 p_2+19200p_3)|\leq 362968,\quad   |p_2^2( 9000 p_2+26100p_1^2)| \leq 489600  \label{4 a7cal2}
\end{equation}
\begin{equation}
  |p_1(33120p_2p_3-69120p_5)|\leq 276480,  \quad 57600 |p_3^2|\leq 230400, \quad \text{and} \quad 108000 |p_2 p_4|\leq 432000.\label{4 a7cal3}
\end{equation}
After inserting the values obtained from equations (\ref{4 a7cal2}) and (\ref{4 a7cal3}) in equation (\ref{4 a7cal}), we arrive at the desired bound of $|a_7|$ as
\begin{equation*}
    |a_7|\leq \frac{1791448}{2073600}\approx 0.863931.
\end{equation*}
\end{proof}

\begin{theorem}
Let $f\in \mathcal{S}^{*}_{\rho}$. Then
\begin{equation*}
    |H_{4,1}(f)|\leq 0.428001.
\end{equation*}
\end{theorem}

\begin{proof}
By substituting the values obtained from Theorem \ref{4 sharpbounds}, Lemma \ref{4 a6a7bound}, and equations (\ref{4 omega1value})-(\ref{4 omega3value}) in equation (\ref{4 4hankelexpression}), we get the desired result.
\end{proof}

%On Acknowledgment %%%%%%%%
	%		\section*{Acknowledgement}
	%		The author would like to thank the referees for the helpful
	%		suggestions.
	%		
	%%%% Bibliography  %%%%%%%%%%
	\subsection*{Acknowledgment}
	Neha Verma is thankful to the Department of Applied Mathematics, Delhi Technological University, New Delhi-110042 for providing a Research Fellowship.

\end{document}